\documentclass[11pt,tbtags]{amsart}
\usepackage{amssymb}
\usepackage{amsmath}
\usepackage{amsthm}

\usepackage{enumitem}
\usepackage{dsfont}

\usepackage[swedish, english]{babel}
\usepackage[latin1]{inputenc}
\usepackage{psfrag}   
\usepackage{graphicx}
\usepackage{subfigure}
\usepackage[square, comma, numbers, sort&compress]{natbib}
\usepackage{enumitem}
\usepackage{color}

\usepackage{caption}
\usepackage[percent]{overpic}
\usepackage{pict2e}

\theoremstyle{plain}

\newcommand{\E}{\mathbb E}

\newcommand{\F}{\mathcal F}
\newcommand{\D}{\mathcal D}
\newcommand{\C}{\mathcal C}

\def\P{{\mathbb P}}

\newtheorem{theorem}{Theorem}[section]

\newtheorem{proposition}[theorem]{Proposition}

\newtheorem{assumption}[theorem]{Assumption}
\newtheorem{remark}[theorem]{Remark}

\title{Multi-dimensional sequential testing and detection}
\author{Erik Ekstr\"om and Yuqiong Wang}
\address{Department of Mathematics, Uppsala University, Box 480, 75106 Uppsala, Sweden.}
\date{\today}

\begin{document}

\begin{abstract}
We study extensions to higher dimensions of the classical Bayesian sequential testing and detection problems for Brownian motion.
In the main result we show that, for a large class of problem formulations, the cost function is unilaterally concave. 
This concavity result is then used to deduce structural properties for the continuation and stopping regions in specific examples.
\end{abstract}

\maketitle

\section{Introduction}

In the seminal paper \cite{S}, two sequential problems of determining an unknown drift of a one-dimensional Wiener process were solved. To describe these problems, let $Y$ be a continuous-time Markov chain with state space $\{0,1\}$ and
transition rate matrix 
\[Q=\begin{pmatrix}
-\lambda & \lambda \\
0 & 0
\end{pmatrix}\]
where $\lambda\geq 0$ is a known constant,
and with random starting point such that $\P(Y_0=1)=\pi$ and $\P(Y_0=0)=1-\pi$ for $\pi\in[0,1]$.
Moreover, let $X$ be a stochastic process given by
\[X_t=\mu \int_0^t Y_s\,ds +W_t,\]
where $\mu\not=0$ and $W$ is a one-dimensional standard Brownian motion independent of $Y$. 
In \cite{S}, the problem of sequential testing between two hypotheses and the problem of quickest detection of a drift change are solved.

{\bf Sequential testing:} In this problem, $\lambda=0$ so that $Y_t=Y_0$ for all $t\geq 0$, and one seeks to determine the 
drift $ \mu Y_0$ as accurately as possible but also as quickly as possible.
More precisely, 
for constants $a,b,c\in(0,\infty)$, \cite{S} considers the problem 
\begin{equation}\label{1dst}
\inf_{\tau,d}\left\{a\P(d=0, Y_0=1) + b\P(d=1, Y_0=0) + c\E[\tau]\right\},
\end{equation}
where the infimum is taken over $\mathcal F^X$-stopping times $\tau$ and decisions $d\in\{0,1\}$ such that $d$ is 
$\mathcal F^X_\tau$-measurable. 

{\bf Quickest detection:} In this problem, $\lambda>0$, and one seeks to detect the jump-time of $Y$ as quickly as possible. More precisely, for $b,c\in(0,\infty)$, \cite{S} considers the problem
\begin{equation}\label{1dqd}
\inf_{\tau}\left\{ b\P(Y_\tau=0) + c\E\left[\int_0^\tau Y_t\,dt\right]\right\},
\end{equation}
where the infimum is taken over $\mathcal F^X$-stopping times $\tau$.

By standard methods, both problems \eqref{1dst} and \eqref{1dqd} can be reduced to optimal
stopping problems of the form
\begin{equation}
\label{1d-form}
\inf_\tau\E\left[g(\Pi_\tau) + \int_0^\tau h(\Pi_s)\,ds\right]
\end{equation}
written in terms of the conditional probability process 
\[\Pi_t:=\E[Y_t\vert\mathcal F_t],\] 
where $g$ and $h$ are certain penalty function (for the sequential testing problem, 
$g(\pi)=a\pi\wedge b(1-\pi)$ and $h(\pi)=c$, whereas in the detection problem, $g(\pi)=b(1-\pi)$ and $h(\pi)=c\pi$). Moreover, it is well-known that the conditional probability process $\Pi$ is a Markov process with generator
\[\frac{1}{2}\pi^2(1-\pi)^2\partial_\pi^2 +\lambda(1-\pi)\partial_\pi.\]
In \cite{S}, problems \eqref{1dst} and \eqref{1dqd} are solved separately using a 
'guess and verify' approach involving an associated free-boundary problem for the cost function. 

The two examples above are the generic formulations in the one-dimensional case, and there is a rich literature on various extensions. For example, testing and detection problems for a Poisson process with
unknown intensity have been studied in \cite{PS1} and \cite{PS2}, and a multi-source variant has been studied in \cite{DPS}.
The references \cite{EV} and \cite{EV2} treat some aspects of testing and detection problems with 
general distributions of the random drift; since the penalties for a wrong decision in \cite{EV} and \cite{EV2} are binary, the sufficient statistic is a one-dimensional, but time-inhomogeneous, Markov process.
Formulations allowing for non-binary penalties appear in \cite{BDK}, \cite{MS} and \cite{ZS}, in which the natural sufficient statistic is two-dimensional and the analysis thus becomes more involved.

In the literature cited above, the observation process is one-dimensional; the existing literature on 
multi-dimensional versions is sparser. In \cite{EPZ}, a three-dimensional Brownian motion is observed for which
exactly one coordinate has non-zero drift, and the problem of determining this coordinate as quickly as possible is studied.
In the set-up of \cite{EPZ}, the three random drifts are heavily dependent; in fact, 
if one drift is non-zero then the remaining two drifts have to be zero.
In \cite{BP} a less constrained set-up is used, in which two Poisson processes change intensity at two independent exponential times, and the problem of detecting the minimum of these two times is considered.

In the current article, we use a similar unconstrained set-up as in \cite{BP} to study sequential testing and detection 
problems for a multi-dimensional Wiener process.
The variety of possible versions of such testing and detection problems is very rich;
indeed, in some applications it would be natural to 
seek to determine {\em all} drifts as accurately as possible, whereas it would be more natural in other applications to 
determine only {\em one} of all possible drifts. Similarly, in the quickest detection problem some applications would 
suggest to look for the {\em smallest} change-point (as in \cite{BP}), whereas one in other applications would try to detect
the {\em last} change-point; further variants are listed below. 
Rather than studying all different formulations on a case by case basis, 
the multitude of multi-dimensional formulations motivates a unified treatment of the corresponding stopping problems.
It turns out that a large class of such problems can be written in the form \eqref{1d-form} (or rather, a multi-dimensional version of \eqref{1d-form}), with $g$ and $h$ both unilaterally concave (concave in each variable separately). In our main result we 
show that unilateral concavity of the penalty functions is preserved in the sense that also the corresponding cost function is unilaterally concave. Since many multi-dimensional penalty functions are unilaterally piecewise affine, the concavity property provides valuable information about the structure of the corresponding continuation and stopping regions.

There is related literature on preservation of spatial concavity/convexity (and consequences for volatility mis-specification) for martingale diffusions within the mathematical finance literature, 
see for example \cite{EJS}, \cite{H} and \cite{JT} for one-dimensional results. In higher dimensions, preservation of concavity is a rather rare property, compare \cite{JT2} and \cite{EJT}. With this in mind, we point out that preservation of {\em unilateral} concavity is a weaker property; however, it is of less financial importance, and has therefore been less studied in the financial literature. 
Also note that for the multi-dimensional version of \eqref{1d-form}, the natural choices of $g$ and $h$ are typically not concave, but only unilaterally concave. We also remark that the authors of \cite{BP} use a three-dimensional embedding of a detection problem in order to obtain concavity of the value function; for {\em unilateral} concavity, however, one may remain in the two-dimensional set-up of the problem.

The paper is organised as follows.
In Section~\ref{2} we specify the multi-dimensional versions of the sequential 
testing and quickest detection problems, and we provide a list of natural examples. In Section~\ref{3} we provide our unilateral concavity result for the multi-dimensional problem, and in Sections~\ref{4}-\ref{5} we use the unilateral concavity to derive structural properties of
continuation regions for the specific examples.

\section{The multi-dimensional set-up}
\label{2}

We consider a problem where one continuously observes an $n$-dimensional process $X$, and where the drift of each component $X^i$ is modeled using a 
continuous time Markov chain $Y^i$ with state space $\{0,1\}$ and transition rate matrix
\[Q^i=\begin{pmatrix}
-\lambda_i & \lambda_i \\
0 & 0
\end{pmatrix},\]
with $\lambda_i\geq 0$. Moreover, the initial condition satisfies 
$\mathbb P(Y_0=1)=\pi_i\in[0,1]$.
The observation process $(X_t)_{t \geq 0}=(X_t^1,X_t^2,\dots, X_t^n)_{t \geq 0}$ is assumed to be given by
\[d X_t^i=\mu_i Y^i_t dt+dW_t^i.\]
Here $\mu_i>0$, $i=1,...,n$ are known constants and 
$W^i,\dots, W^n$ are one-dimensional standard Brownian motions such that $Y^1,...,Y^n,W^1,...,W^n$ are independent.
In parallel to the one-dimensional case, we introduce the multi-dimensional posterior probability process $\Pi=(\Pi^1, \dots\Pi^n)$ by
\[\Pi^i_t:=\E[Y^i_t\vert\mathcal F_t^X].\]
By our independence assumption we note that $\Pi^i_t=\E[Y^i_t\vert\mathcal F_t^{X^i}]$, and, in particular, 
that the coordinates of $\Pi$ are independent.

Below we list a few natural formulations of multi-dimensional sequential testing problems and multi-dimensional detection problems.
These examples can be written as stopping problems of the form
\begin{equation}
\label{multid}
\inf_\tau\E\left[g(\Pi_\tau) + \int_0^\tau h(\Pi_s)\,ds\right]
\end{equation}
as in the one-dimensional case, but 
with $g,h:[0,1]^n\to[0,\infty)$ now being functions of the multi-dimensional process $\Pi$.

{\bf Sequential testing.} 
Assume that $\lambda_i=0$, $i=1,...,n$ and that the penalisation in time is linear, i.e. of the type
$c\E[\tau]$ for some constant $c>0$. All formulations below can then be written on the form \eqref{multid} with $h=c$ 
but with different penalty functions $g$. For simplicity we consider symmetric penalization (corresponding to $a=b=1$ in \eqref{1dst}); generalizations to set-ups with non-symmetric weights 
are straightforward.


\begin{itemize}
\item[(ST1)]
Consider the problem
\begin{equation}\notag
\inf_{\tau,d}\left\{\sum_{i=1}^n \mathbb P(d_i\not=Y^i_0)+ c\E\left[\tau\right]\right\},
\end{equation}
where the infimum is taken over $\mathcal F^X$-stopping times $\tau$ and decisions $d\in\{0,1\}^n$ such that $d$ is $\mathcal F^X_\tau$-measurable, i.e. 
the tester is penalised equally for every faulty decision. This problem can be written on the form \eqref{multid}
with  
\[g(\pi)=\sum_{i=1}^n \pi_i\wedge (1-\pi_i).\]
\item[(ST2)]
Consider the problem 
\[\inf_{\tau,d,\tilde d}\left\{\P(d\not=Y^{\tilde d}_0) + c\E[\tau]\right\},\]
where the infimum is taken over $\mathcal F^X$-stopping times $\tau$ and decisions $d\in\{0,1\}$ and 
$\tilde d\in\{1,\dots,n\}$ that are $\mathcal F_\tau^X$-measurable. Thus the tester seeks to determine as quickly as possible a drift for only one of the processes. For this problem, the penalty function is given by
\[g(\pi)=\bigwedge_{i=1}^n(\pi_i\wedge(1-\pi_i)).\]
\item[(ST3)] 
Let for simplicity $n=2$ (generalizations are straightforward), let $\mu_1=\mu_2=:\mu$ and let $\gamma\in[0,1]$ be a given constant. Consider the problem 
\[\inf_{\tau_1,\tau_2,d_1,d_2}\left\{ \P(d_1\not=Y^1_{0}) + \P(d_2\not= Y^2_0)+ 
c\E[\tau_1\wedge\tau_2+(1-\gamma)(\tau_1\vee\tau_2-\tau_1\wedge\tau_2)]\right\},\]
where the infimum is taken over stopping times $\tau_1$, $\tau_2$ and decisions $d_1,d_2\in\{0,1\}$ such that 
$d_i$ is $\mathcal F_{\tau_i}^X$-measurable.
Here $\gamma$ is a cost reduction parameter which describes how the cost (per unit of time) of observing two processes relates to the cost of observing only one process. Using the strong Markov property, this multiple stopping problem reduces to a problem of type \eqref{multid} with penalty
\[g(\pi_1,\pi_2)=\bigwedge_{i=1}^2 \left(\pi_i\wedge(1-\pi_i) + u^{\mu,c(1-\gamma)}(\pi_{3-i})\right),\]
where $u^{\mu,c(1-\gamma)}$ is the value function of the one-dimensional sequential testing problem with cost $c(1-\gamma)$ per unit of time, see \eqref{1d} below.
\end{itemize}

{\bf Quickest detection.}
Now assume that $\lambda_i>0$ for all $i$, and let $c>0$ be a constant.
In all of the formulations below, the infimum is taken over $\mathcal F^X$-stopping times.

\begin{itemize}
\item[(QD1)]
Consider the problem
\[\inf_{\tau}\left\{\mathbb P\left(\max_{1\leq i\leq n}Y^i_\tau =0\right) + c \E\left[\int_0^\tau \max_{1\leq i\leq n}Y^i_t\,dt\right]\right\}.\]
Here one seeks to determine the {\em first} change-point (this problem formulation was treated in \cite{BP} for a detection problem involving two Poisson processes);
the problem can be written on the form \eqref{multid} with 
\[g(\pi)=\Pi_{i=1}^n(1-\pi_i)\] 
and 
\[h(\pi)=c(1-\Pi_{i=1}^n(1-\pi_i)).\]
\item[(QD2)]
A problem of determining the {\em last} change-point is obtained by instead considering
\[\inf_{\tau}\left\{\mathbb P\left(\min_{1\leq i\leq n}Y^i_\tau =0\right) + c \E\left[\int_0^\tau \min_{1\leq i\leq n}Y^i_t\,dt\right]\right\}.\]
Again, this problem can be written on the form \eqref{multid}; the corresponding functions $g$ and $h$ are given by 
\[g(\pi)=1-\Pi_{i=1}^n\pi_i\] 
and 
\[h(\pi)=c\Pi_{i=1}^n\pi_i,\]
respectively.
\item[(QD3)]
Assume that a tester wants to detect one coordinate for which the change-point has happened. One possible formulation of this is
\[\inf_{\tau,\tilde d}\left\{\P(Y^{\tilde d}_\tau=0) + c\E\left[\int_0^\tau \sum_{i=1}^n Y^i_t  \,dt\right]\right\},\]
where the infimum is taken also over $\F^X_\tau$-measurable decisions $\tilde d\in\{1,...,n\}$.
The problem can be written on the form \eqref{multid} where the
corresponding functions $g$ and $h$ are given by
\[g(\pi)=\bigwedge_{i=1}^n(1-\pi_i)\] 
and 
\[h(\pi)=c\sum_{i=1}^n\pi_i.\]
\end{itemize}

Now consider the stopping problem \eqref{multid} for given functions $g$ and $h$.
Throughout the remainder of this article we make the following assumption.

\begin{assumption}
\label{ass}
We assume that
\begin{itemize}
\item
the functions $g,h:[0,1]^n\to[0,\infty)$ are Lipschitz continuous;
\item
the functions $g(\pi)$ and $h(\pi)$ are concave in each variable $\pi_i$ separately;
\item
if $\lambda_i=0$ for some $i=1,\dots,n$, then $h$ is constant.
\end{itemize}
\end{assumption}

\begin{remark}
Note that all examples (ST1)-(ST4) and (QD1)-(QD3) are covered by Assumption~\ref{ass}. 
Also note that the assumption of unilateral concavity is strictly weaker than (joint) concavity. In fact, 
$g$ and $h$ are {\em not} concave in (QD1)-(QD2).
\end{remark}

It is well-known that the $\Pi$ process satisfies
\begin{equation}
\label{Pi}
d \Pi_t^i=\lambda_i (1- \Pi_t^i) dt+\mu_i \Pi_t^i (1- \Pi_t^i)d \bar{W}_t^i
\end{equation}
for $i=1,\dots, n$, where the innovation process $\bar{W}=(\bar{W}^1,...,\bar{W}^n)$ defined by 
\[\bar W^i_t=X^i_t-\mu_i\int_0^t\Pi^i_s\,ds\]
is an $n$-dimensional Brownian motion with independent coordinates.
Consequently, $\Pi$ is an $n$-dimensional time-homogeneous Markov process with independent coordinates; allowing for an arbitrary starting point $\pi\in[0,1]^n$, we define a cost function $V:[0,1]^n\to[0,\infty)$ by
\begin{equation}
\label{V}
V(\pi)=\inf_{\tau}\mathbb{E}_{\pi}\left[g(\Pi_{\tau})+\int_0^{\tau}h(\Pi_s)ds\right].
\end{equation}
We also introduce the continuation region 
\[\C:=\{\pi\in[0,1]^n: V(\pi)< g(\pi)\}\]
and its complement, the stopping region $\D=[0,1]^n\setminus\C$, 
and we recall from optimal stopping theory that the stopping time 
\[\tau^*=\inf\{t \geq 0: \Pi_t \in\D\}\]
is optimal in \eqref{V}.

We end this section with a short review of the one-dimensional problems.

\subsection{The one-dimensional case}

\subsubsection{Sequential testing}
With the notation of the introduction, let 
\begin{equation}
\label{1d}
u(\pi):=u^{\mu,c}(\pi):=\inf_\tau\E_\pi[\Pi_\tau\wedge(1-\Pi_\tau) +c\tau].
\end{equation}
The notation $u^{\mu,c}$ is used when we want to emphasize the dependence on the drift $\mu$ and the cost of observation parameter $c$, and we refer to this one-dimensional testing problem as $ST(\mu,c)$.
We then know that $u:[0,1]\to[0,1]$ is concave with $u(\pi)\leq\pi\wedge(1-\pi)$. Moreover, 
\[\C:=\{\pi\in[0,1]:u(\pi)<\pi\wedge(1-\pi)\} = (A^*,1-A^*)\]
for some $A^*\in(0,1/2)$; further details on $u$ and $A^*$ can be found in \cite{PS}. 

\subsubsection{Quickest detection}
Again with the notation of the introduction, let
\[u(\pi):=u^{\mu,\lambda,c}(\pi_):=\inf_\tau\E_\pi\left[1-\Pi_\tau+c\int^\tau_0\Pi_t\,dt\right],\]
where the notation $u^{\mu,\lambda,c}$ is used when we want to emphasize the dependence on the parameters $\mu$, $\lambda$ and $c$, and we refer to this one-dimensional detection problem as $QD(\mu,\lambda,c)$.
The function $u:[0,1]\to[0,1]$ is then concave and non-increasing. Moreover, 
\[\C:=\{\pi\in[0,1]:u(\pi)<1-\pi\} = [0,B^*)\]
for some $B^*\in (0,1)$; again, further details on $u$ and $B^*$ can be found in \cite{PS}.

\section{Properties of the cost function}
\label{3}

In this section we derive Lipschitz continuity and unilateral concavity for the multi-dimensional stopping problem
\eqref{multid}.

\subsection{Continuity}

\begin{theorem}
The cost function $V:[0,1]^n\to[0,\infty)$ is Lipschitz continuous.
\end{theorem}

\begin{proof}
It suffices to check that $V$ is Lipschitz in each variable $\pi_i$ separately. To do that, let $i=1$ and
denote by $\Pi_t$ the solution of \eqref{Pi} with initial condition $\Pi_t=\pi\in[0,1]^n$, and denote by 
$\tilde\Pi_t$ the solution with initial condition $\tilde\pi=(\tilde\pi_1,...,\tilde\pi_n)$, where $\pi_j=\tilde\pi_j$, $j=2,\dots,n$ and $\pi_1<\tilde\pi_1$. By a comparison result for one-dimensional 
stochastic differential equations, $\Pi_t^1\leq\tilde\Pi_t^1$ for all $t\geq 0$.
Moreover,  
\[d(\tilde\Pi^1_t-\Pi^1_t) = -\lambda_1 (\tilde\Pi^1_t-\Pi^1_t)\,dt + dM_t\]
where $M$ is a continuous martingale, so $\tilde\Pi^1_t-\Pi^1_t$ is a bounded 
supermartingale. Consequently, by optional sampling, 
\[0\leq \E[ \tilde\Pi^1_{\tau}-\Pi^1_{\tau}]\leq \tilde\pi_1-\pi_1\]
for any stopping time $\tau$.
Thus 
\[\vert \E[g(\tilde\Pi_\tau)-g(\Pi_\tau)]\vert\leq C\E[\vert \tilde\Pi_\tau -\Pi_\tau\vert]\leq C\vert\tilde\pi_1-\pi_1\vert\]
where $C$ is a Lipschitz constant of $g$. This shows that if $h$ is constant, then $V$ is Lipschitz
in its first argument, and thus it is Lipschitz also in $\pi$.

Furthermore, if $\lambda_1>0$, then 
\begin{eqnarray*}
\E[ \tilde\Pi^1_{t}-\Pi^1_{t}]&=& \tilde\pi_1-\pi_1 -\lambda_1\int_0^{t} \E\left[\tilde\Pi^1_{s}-\Pi^1_{s}\right]\,ds,
\end{eqnarray*}
so 
\[\mathbb{E}[\tilde\Pi^1_{t}-\Pi^1_{t}]=(\tilde\pi_1-\pi_1)e^{-\lambda_1 t}.\]
Consequently, 
\begin{eqnarray*}
\E\left[\int_0^\tau \vert h(\tilde\Pi_t)-h(\Pi_t)\vert\,dt\right] \leq 
D \int_0^\infty \E\left[\vert \tilde\Pi_t-\Pi_t\vert\right]\,dt =\frac{D}{\lambda_1}(\tilde\pi_1-\pi_1),
\end{eqnarray*}
where $D$ is a Lipschitz constant of $h$.
It follows that
\begin{eqnarray*}
\left\vert\E\left[g(\tilde\Pi_\tau)+\int_0^\tau h(\tilde\Pi_t)\,dt\right] -
\E\left[g(\Pi_\tau)+\int_0^\tau h(\Pi_t)\,dt\right]\right\vert \leq
(C+\frac{D}{\lambda_1})(\tilde \pi_1-\pi_1)
\end{eqnarray*}
for any stopping time $\tau$.
Therefore $V(\pi_1,\dots,\pi_n)$ is Lipschitz in its first argument. Consequently, if $\lambda_i>0$ for all 
$i=1,...,n$, then $V$ is Lipschitz also in $\pi$.
\end{proof}

\begin{remark}\label{rem}
It follows from the proof above that if $\lambda_i=0$ for $i=1,\dots,n$, 
$g$ is Lipschitz 1 in each variable separately and $h$ is constant, then 
also $V$ is Lipschitz 1 in each variable separately.
\end{remark}

\subsection{Unilateral concavity}

Next we study unilateral concavity of the value function.

Let $\tilde\P$ be a new measure defined so that 
\[\left.\frac{d\tilde \P}{d \P} \right\vert_{\F^X_t}=
\exp\left\{- \frac{1}{2}\sum_{i=1}^n\mu_i^2\int_0^t (\Pi^i_s)^2\,ds -\sum_{i=1}^n\int_0^t\mu_i\Pi^i_s\,d\bar W^i_s\right\},\]
and denote by $\tilde \E$ the corresponding expectation operator.
By the Girsanov theorem, $X_t$ is an $n$-dimensional $\tilde \P$-Brownian motion. 
Define the probability likelihood process $\Phi=(\Phi^1,...,\Phi^n)$ by
\[\Phi^i_t=\frac{\Pi^i_t}{1-\Pi^i_t}\]
and observe that $\Phi_0^i=\frac{\pi_i}{1-\pi_i}=:\phi_i$. Also note that an application of Ito's formula yields
\begin{equation}
\label{eq_phi}
d\Phi^i_t=\lambda_i(1+\Phi^i_t)dt+\mu_i\Phi^i_tdX_t^i.
\end{equation}

\begin{proposition}
\label{prop}
We have 
\[\frac{V(\pi)}{\prod_{i=1}^n(1-\pi_i)}=\inf_{\tau}\tilde\E\left[ e^{-\lambda \tau}\left(\prod_{i=1}^n(1+\Phi_\tau^i)\right) g(\Pi_\tau)
+\int_0^\tau e^{-\lambda t}\left(\prod_{i=1}^n(1+\Phi_t^i)\right)h(\Pi_t)\,dt\right], \]
where $\lambda=\sum_{i=1}^n\lambda_i$.
\end{proposition}

\begin{proof}
Define a process $Y_t$ by
 \[Y_t=e^{-\lambda t}\prod_{i=1}^n \frac{1-\pi_i}{1-\Pi_t^i} = e^{-\lambda t}\prod_{i=1}^n (1+\Phi^i_t)(1-\pi_i)\]
and observe that $Y_0=1$.
Using Ito's formula and \eqref{Pi} we find that
 \begin{eqnarray*}
 dY_t &=& -\lambda Y_t dt+Y_t\sum_{i=1}^n \frac{1}{1-\Pi_t^i} d\Pi_t^i +Y_t\sum_{i=1}^n \frac{1}{(1-\Pi_t^i)^2} (d\Pi_t^i )^2\\
 &=& -\lambda Y_t dt+Y_t\sum_{i=1}^n (\lambda_idt+\mu_i\Pi^i_t d\bar W^i_t) +Y_t\sum_{i=1}^n\mu_i^2(\Pi^i_t)^2dt\\
 &=&Y_t \sum_{i=1}^n\mu_i\Pi^i_t dX_t^i.
 \end{eqnarray*}
Since  
\[\left.\frac{d \P}{d \tilde\P} \right\vert_{\F^X_t}=
\exp\left\{- \frac{1}{2}\sum_{i=1}^n\mu_i^2\int_0^t (\Pi^i_s)^2\,ds -\sum_{i=1}^n\int_0^t\mu_i\Pi^i_s\,dX^i_s\right\},\]
it follows that  
\[\left.\frac{d \P}{d\tilde \P} \right\vert_{\F^X_t}=Y_t, \]
so we can rewrite the value function as 
 \begin{eqnarray*}
 \frac{V(\pi)}{\prod_{i=1}^n (1-\pi_i)} &=& \frac{1}{ \prod_{i=1}^n (1-\pi_i)}  \inf_{\tau}\E_{\pi}\left[g(\Pi_{\tau})+\int_0^{\tau}h(\Pi_{t})dt\right]\\
 &=&\frac{1}{ \prod_{i=1}^n (1-\pi_i)} \inf_{\tau}\tilde\E_{\pi}\left[Y_{\tau}\left(g(\Pi_{\tau})+\int_0^{\tau}h(\Pi_{t})dt\right)\right]\\
 &=&\inf_{\tau}\tilde\E_{\pi}\left[e^{-\lambda \tau}\prod_{i=1}^n (1+\Phi_{\tau}^i)g(\Pi_{\tau})+e^{-\lambda \tau}\prod_{i=1}^n (1+\Phi_{\tau}^i)\int_0^{\tau}h(\Pi_{t})dt\right]\\
 &=& \inf_{\tau}\tilde\E_{\pi}\left[e^{-\lambda \tau}\prod_{i=1}^n (1+\Phi_{\tau}^i)g(\Pi_{\tau})+\int_0^{\tau}e^{-\lambda t}\prod_{i=1}^n (1+\Phi_{t}^i)h(\Pi_{t})dt\right],
 \end{eqnarray*}
which completes the proof.
\end{proof}

\begin{theorem}
The function $\pi_i\mapsto V(\pi)$ is concave in each variable separately (i.e. $\pi_i\mapsto V(\pi)$ is concave for each $i=1,...,n$).
\end{theorem}

\begin{proof}
It suffices to check that $\pi_1\mapsto V(\pi)$ is concave. To do that, first note that since \eqref{eq_phi} 
is a linear equation, it can be solved explicitly as
\begin{equation}
\label{expsolphi}
\Phi^i_t=e^{(\lambda_i-\frac{\mu_i^2}{2})t+\mu_i X^i_t}\left(\phi_i+\lambda_i\int_0^t e^{-(\lambda_i-\frac{\mu_i^2}{2})s-\mu_i X^i_s}ds\right).
\end{equation}
Thus $\Phi^i_t$ is affine in $\phi_i$ and independent of $\phi_j$, $j\not= i$.
Moreover, denoting
\[Z_t=e^{(\lambda_1-\frac{\mu_ 1^2}{2})t+\mu_1 X^1_t},\]
we have that 
\[\frac{\partial \Phi^1_t}{\partial \pi_1}=\frac{Z_t}{(1-\pi_1)^2}\]
and
\[\frac{\partial^2 \Phi^1_t}{\partial \pi_1^2}=\frac{2Z_t}{(1-\pi_1)^3}\]
so that
\begin{equation}
\label{rel}
2\frac{\partial \Phi^1_t}{\partial \pi_1} =(1-\pi)  \frac{\partial^2 \Phi^1_t}{\partial \pi_1^2}.
\end{equation}

Fix an $\F^X$-stopping time $\tau$; we next claim that 
\begin{equation}
\label{G}
G(\pi,\omega):=(1-\pi_1)(1+\Phi^1_\tau)g(\Pi_\tau)
\end{equation}
is concave in $\pi_1$. To see this, assume that $g$ is $C^2$ (the general case follows by approximation).
Then
\begin{align*}
G_{\pi_1\pi_1}
&=  -2 \frac{\partial}{\partial_ {\pi_1}}\left(\left(1+\Phi_\tau^1\right) g(\Pi_\tau)\right)+(1-\pi_1)\frac{\partial}{\partial_ {\pi_1 \pi_1}}\left(\left(1+\Phi_\tau^1\right) g(\Pi_\tau)\right)\\
&= -2\frac{\partial \Phi^1_{\tau}}{\partial \pi_1} \left( g+\frac{g_{\pi_1}}{1+\Phi_{\tau}^1}\right)  +(1-\pi_1)\left( \frac{\partial^2 \Phi^1_{\tau}}{\partial \pi_1^2}  ( g+\frac{g_{\pi_1}}{1+\Phi_{\tau}^1})
+(\frac{\partial \Phi^1_{\tau}}{\partial \pi_1})^2 \frac{g_{\pi_1\pi_1}}{(1+\Phi^1_{\tau})^3}\right) \\
&=\left(\frac{1-\Pi_{\tau}^1}{1-\pi_1}\right)^3Z_{\tau}^2 g_{\pi_1\pi_1}
\end{align*}
by \eqref{rel}, so $G$ in \eqref{G} is concave in $\pi_1$. 
Taking expectation we have that
\[\pi_1 \mapsto \prod_{i=1}^n (1-\pi_i)\tilde\E\left[  \left( e^{-\lambda \tau}\prod_{i=1}^n (1+\Phi_{\tau}^i)g(\Pi_{\tau}) \right)\right]\]
is concave in $\pi_1$. By similar arguments,
\[\pi_1\mapsto \int_0^{\tau}e^{-\lambda t}\left(\prod_{i=1}^n(1+\Phi_t^i)\right)h(\Pi_t^i) dt\]
is concave, so 
\[ \prod_{i=1}^n (1-\pi_i)\tilde\E\left[  \left( e^{-\lambda \tau}\prod_{i=1}^n (1+\Phi_{\tau}^i)g(\Pi_{\tau}) \right)
+\int_0^{\tau}e^{-\lambda t}\left(\prod_{i=1}^n(1+\Phi_t^i)\right)h(\Pi_t^i) dt\right]\]
is concave in $\pi_1$ for each stopping time $\tau$.
Taking infimum over stopping times $\tau$, it follows that $\pi_1\mapsto V(\pi)$ concave, which completes the proof.
\end{proof}

\section{Sequential testing problems}
\label{4}

In this section we use the general results of Section~\ref{3} to provide structural results for the multi-dimensional sequential testing problems (ST1)-(ST3). For the sake of graphical illustrations, we present the results for the 
case $n=2$;  the higher-dimensional version works similarly, and our results easily carry over to that case.

\begin{remark}
In the structural studies of (ST1)-(ST3) and (QD1)-(QD3) below, we focus on what conclusions can be drawn from 
our main result on unilateral concavity. 
Refined studies would aim at further properties of the stopping boundaries that we find. For example, a lower
bound on the continuation region is provided by the set $\{\mathcal L g+h< 0\}$, where $\mathcal L$ is the generator of $\Pi$, methods to prove that wedges of $g$ are automatically contained in the continuation region can be obtained, and studies of continuity of stopping boundaries can 
be performed along the lines of \cite{DA}. 
\end{remark}

\subsection{(ST1)}
In this section we provide further details for the problem (ST1) of determining all (i.e. both) drifts. 
Thus we consider the stopping problem 
\[V(\pi)=\inf_{\tau}\E_\pi\left[g(\Pi_\tau) + c\tau\right],\]
where $g(\pi)=\pi_1\wedge(1-\pi_1) +\pi_2\wedge(1-\pi_2)$.
Denote by $R_1=[0,1/2]\times [0,1/2]$ so that $g=\pi_1+\pi_2$ on $R_1$.
Let $(A^*_i,1-A^*_i)$ be the continuation region for the one-dimensional problem $ST(\mu_i,c)$, $i=1,2$.

\begin{proposition}
\label{monotone}
There exists a non-increasing upper semi-continuous function $b:[0,A^*_1]\to [0,A^*_2]$ such that 
\[\D\cap R_1=\{\pi\in R_1:\pi_1\leq A^*_1, \pi_2\leq b(\pi_1)\} .\] 
\end{proposition}

\begin{proof}
First note that $V(\pi_1,0)=u^{\mu_1,c}(\pi_1)$, so $(\pi_1,0)\in\C$ precisely if 
$\pi_1\in(A^*_1,1-A_1^*)$. Since $V$ is 
Lipschitz continuous with parameter 1 in each direction by Remark~\ref{rem}, and since $g$ has slope 1 in the $\pi_2$-direction
for $\pi_2\leq 1/2$, it follows that 
$\D\cap R_1=\{\pi\in R_1:\pi_1\leq A^*_1, \pi_2\leq b(\pi_1)\}$ for some $b:[0,A^*_1]\to[0,1/2]$.
By a similar argument, starting from $(0,\pi_2)$ instead, it follows that $b\leq A^*_2$ and that $b$ is non-increasing.
Finally, the continuity of $V$ implies that $\C$ is open, and $b$ is thus upper semi-continuous.
\end{proof}

\begin{figure}[htp!]\centering
  \subfigure[]{\includegraphics[width=0.45\textwidth]{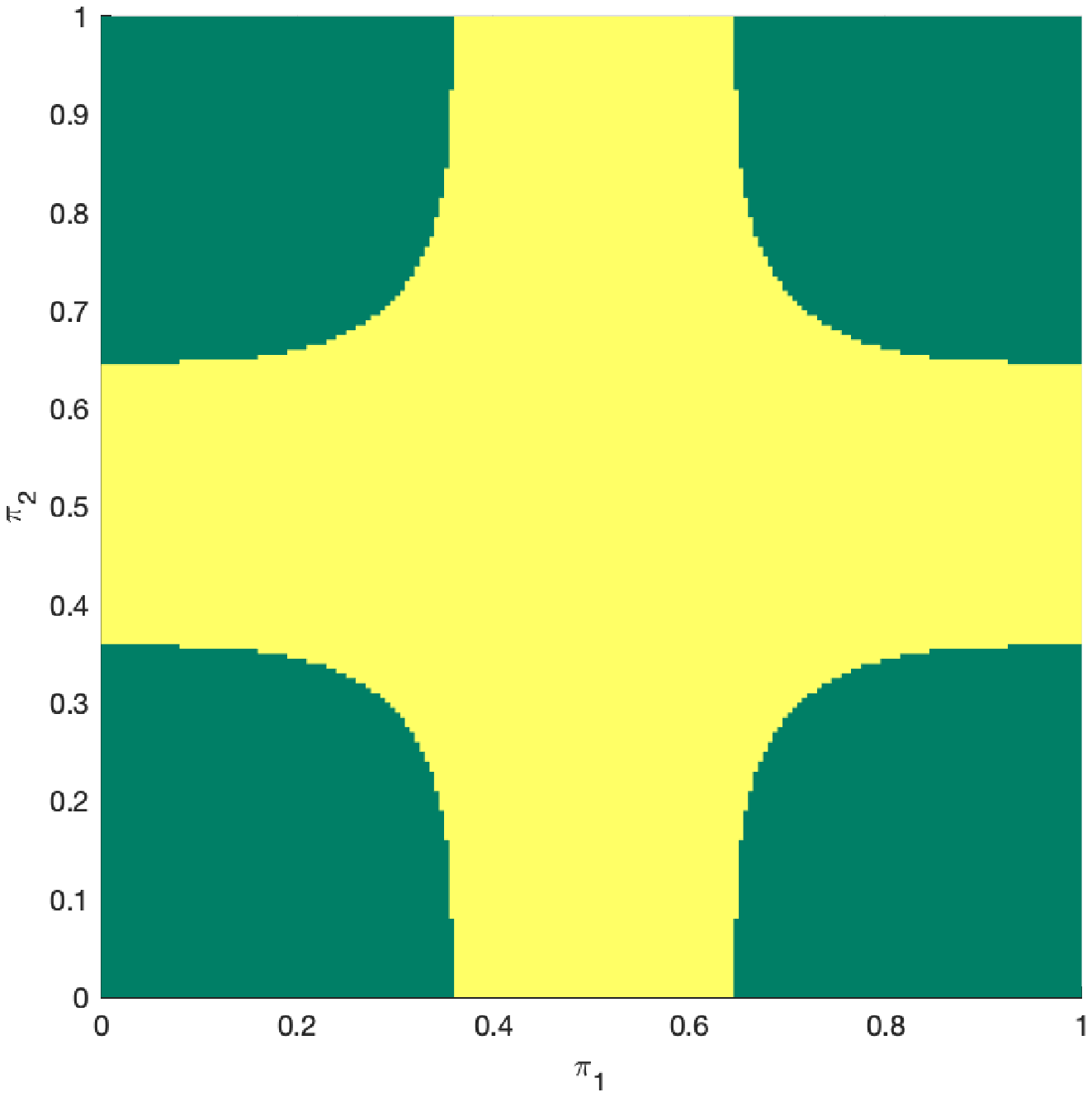}\label{figureST1}}
  \subfigure[]{\includegraphics[width=0.45\textwidth]{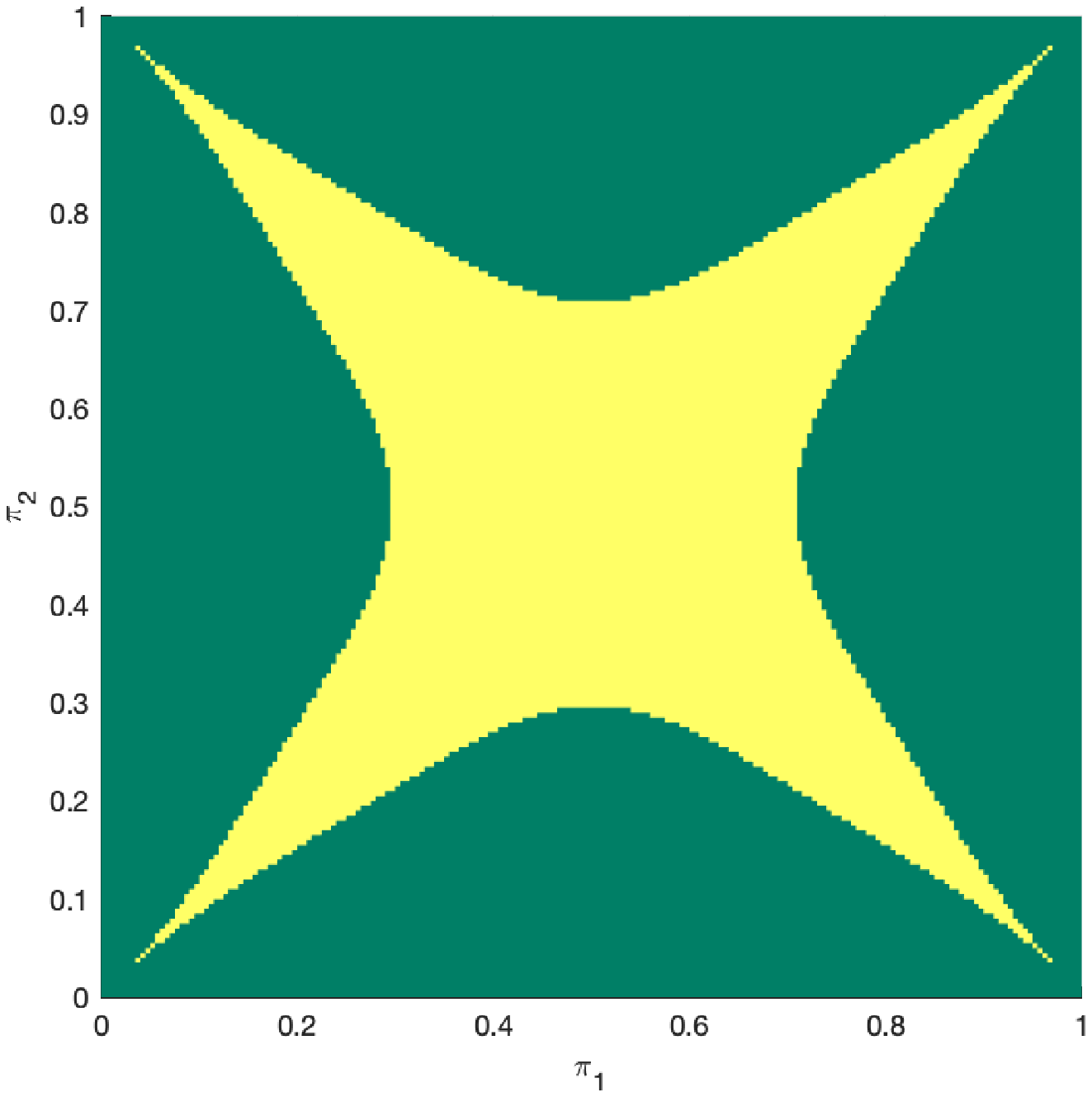}\label{figureST2}}
\caption{Continuation regions (yellow) in (ST1) on the left and in (ST2) on the right.
The parameters are $\mu_1=\mu_2=1$ and $c=0.2$.}
\end{figure}

By symmetry, the part of the continuation region within each square $R_2:=[1/2,1]\times[0,1/2]$, $R_3:=[1/2,1]\times[1/2,1]$ and $R_4:=[0,1/2]\times[0,1/2]$ can be described similarly as in Proposition~\ref{monotone}. For a graphical illustration, see Figure~\ref{figureST1}.

\subsection{(ST2)}
Next we provide further details for the problem (ST2) in the case $n=2$. 
Thus we consider the stopping problem 
\[V(\pi)=\inf_{\tau}\E_\pi\left[g(\Pi_\tau) + c\tau\right],\]
where $g(\pi)=\pi_1\wedge(1-\pi_1) \wedge\pi_2\wedge(1-\pi_2)$. 
Consider the triangular region $T:=\{\pi_2\leq \pi_1\wedge(1-\pi_1)\}$, and note that $g=\pi_2$ in this region.

\begin{proposition}
There exists an upper semi-continuous function $b:[0,1]\to[0,1]$ with $b(\pi_1)\leq\pi_1\wedge(1-\pi_1)$ such that
\[\D\cap T = \{\pi\in T: \pi_2 \leq b(\pi_1)\}.\]
Furthermore, $b$ is non-decreasing on $[0,1/2]$ and satisfies $b(\pi_1)=b(1-\pi_1)$.
\end{proposition}

\begin{proof}
Since $V(\pi_1,0)=0=g(\pi_1,0)$, we clearly have that $(\pi_1,0)$ is in the stopping region. Now, since $V$ is Lipschitz(1) by Remark~\ref{rem}, and since 
$g$ has slope 1 in the $\pi_2$-direction, the existence of $b$ follows. 

The monotonicity property is a consequence of symmetry and concavity: if $(\pi_1,\pi_2)\in T\cap \D$
then also $(1-\pi_1,\pi_2)\in T\cap\D$, so unilateral concavity yields that
the whole line segment $\{(p,\pi_2);\pi_1\leq p\leq 1-\pi_1\}$ belongs to the stopping region.
Finally, the asserted upper semi-continuity of $b$ follows from the continuity of $V$.
\end{proof}

Let $(A_i^*,1-A_i^*)$ be the continuation region of $ST(\mu_i,c)$, $i=1,2$.

\begin{proposition} 
The rectangle $R:=(A_1^* , 1-A_1^*)\times (A_2^* ,1- A_2^*)$ is contained in the continuation region.
\end{proposition}

\begin{proof}
Take $\pi\in R$, and let $i\in\{1,2\}$ be such that $g(\pi)=\pi_i\wedge(1-\pi_i)$.
Define
\[\tau_i:=\inf\{t\geq 0:\Pi^i_t\notin (A_i^*,1-A_i^*)\}\]
to be the optimal stopping time in the one-dimensional problem of determining $Y^i$.
Then 
\[V(\pi)\leq \E_\pi\left[ g(\Pi_{\tau_i})+ c\tau_i\right] \leq \E_\pi\left[ \Pi^i_{\tau_i}\wedge(1-\Pi^i_{\tau_i}) + c\tau_i\right] =u^{\mu_i,c}(\pi_i)<\pi_i\wedge(1-\pi_i)=g(\pi),\]
which shows that $\pi\in\C$. 
\end{proof}

For a graphical illustration of the continuation region in (ST2), see Figure~\ref{figureST2}.

\subsection{(ST3)}

We now study the sequential testing problem (ST3) with cost reduction given by $\gamma\in(0,1)$ (we exclude the cases $\gamma\in\{0,1\}$ since they correspond to (ST1) and (ST2), respectively).
The value function of this problem is 
\[V(\pi)=\inf_{\tau}\E_\pi\left[g(\Pi^1_\tau,\Pi^2_\tau) +c\tau\right],\]
where
\[g(\pi_1,\pi_2):= \left(\pi_1 \land (1-\pi_1)+u(\pi_2)\right) \land \left(u(\pi_1)+ \pi_2 \land (1-\pi_2)\right ) \]
and $u=u^{\mu,c(1-\gamma)}$ is the value function of the one-dimensional problem $ST(\mu,c(1-\gamma))$.
Since $u$ is concave and Lipschitz(1), the value 
function $V$ is also concave and Lipschitz(1) in each variable. 
Denote by $T:= \{\pi\in[0,1]^2:0\leq\pi_2\leq \pi_1\wedge(1-\pi_1)\}$, and note that $g(\pi_1,\pi_2)=u(\pi_1) + 
\pi_2$ on $T$. 

\begin{proposition}
There exists an upper semi-continuous function $b:[0,1]\to[0,1/2]$ with $b(\pi_1)\leq\pi_1\wedge(1-\pi_1)$ such that $\D\cap T=\{\pi\in T:\pi_2\leq b(\pi_1)\}$.
Moreover, $b(\pi_1)=b(1-\pi_1)$.
\end{proposition}

\begin{proof}
We first claim that $[0,1]\times\{0\}\subseteq\D$. To see this, note that $g(\pi_1,0)=u(\pi_1)$ and that 
$u(\Pi^1_t) + ct$ is a submartingale. It follows that $V(\pi_1,0)=u(\pi_1)$.

Next, the fact that $g$ is affine in $\pi_2$ on $T$ together with unilateral concavity of $V$ give the existence of $b$.
The upper semi-continuity of $b$ follows from continuity of $V$, and the symmetry of $b$ follows from the symmetric set-up.
\end{proof}

\begin{figure}[htp!]\centering
  \subfigure[]{\includegraphics[width=0.3\textwidth]{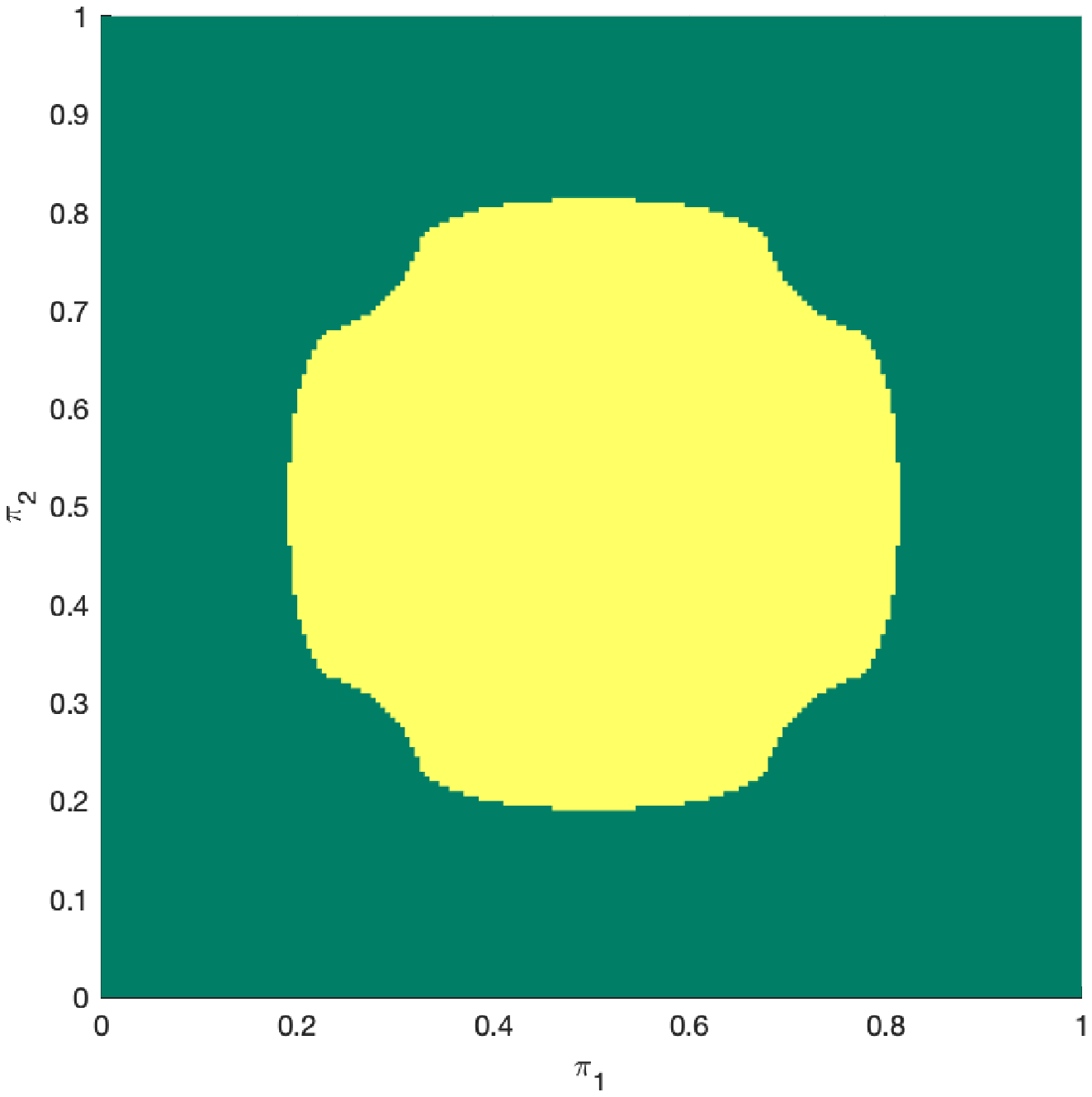}\label{figureST3-1}}
  \subfigure[]{\includegraphics[width=0.3\textwidth]{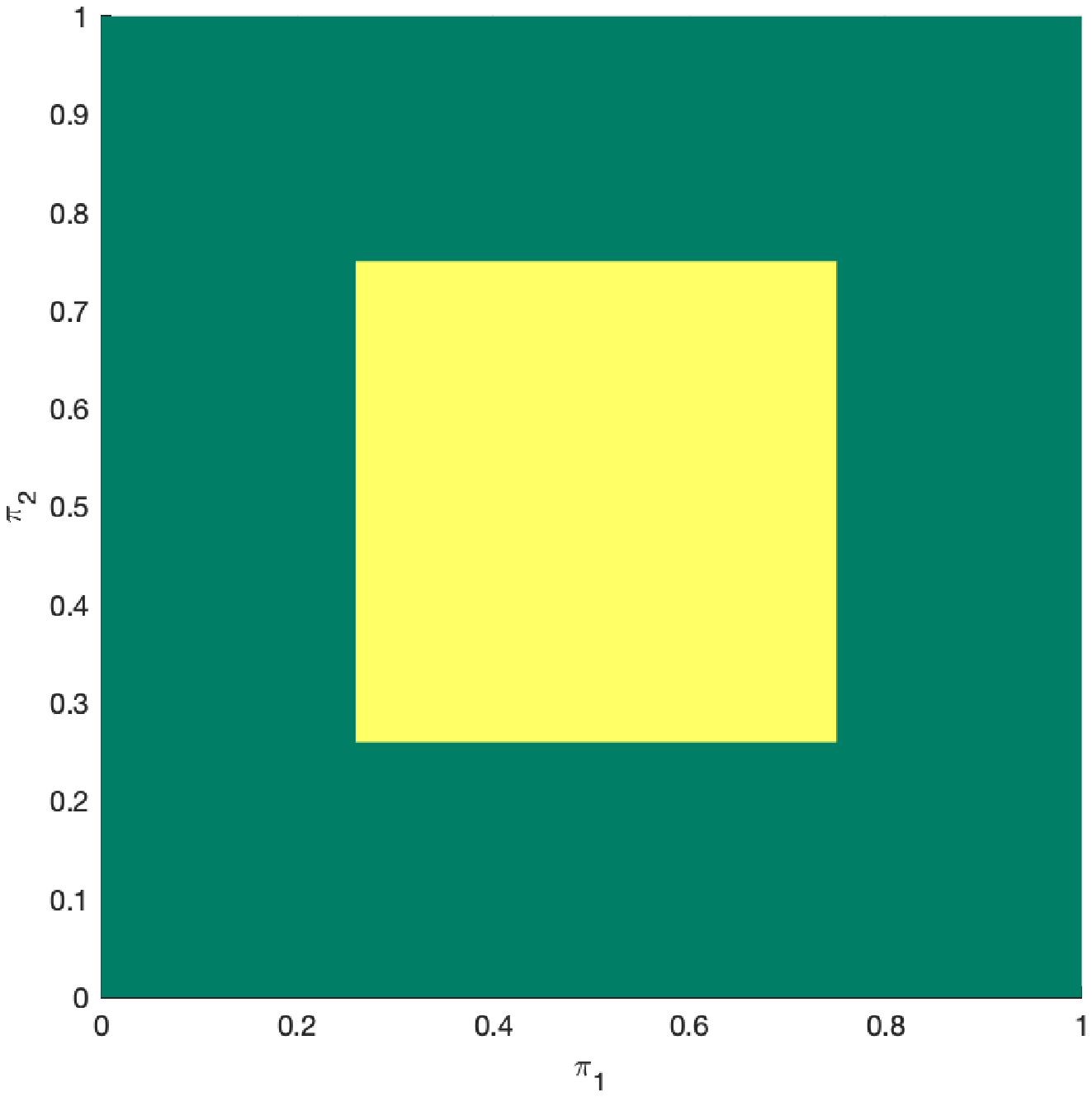}\label{figureST3-2}}
\subfigure[]{\includegraphics[width=0.3\textwidth]{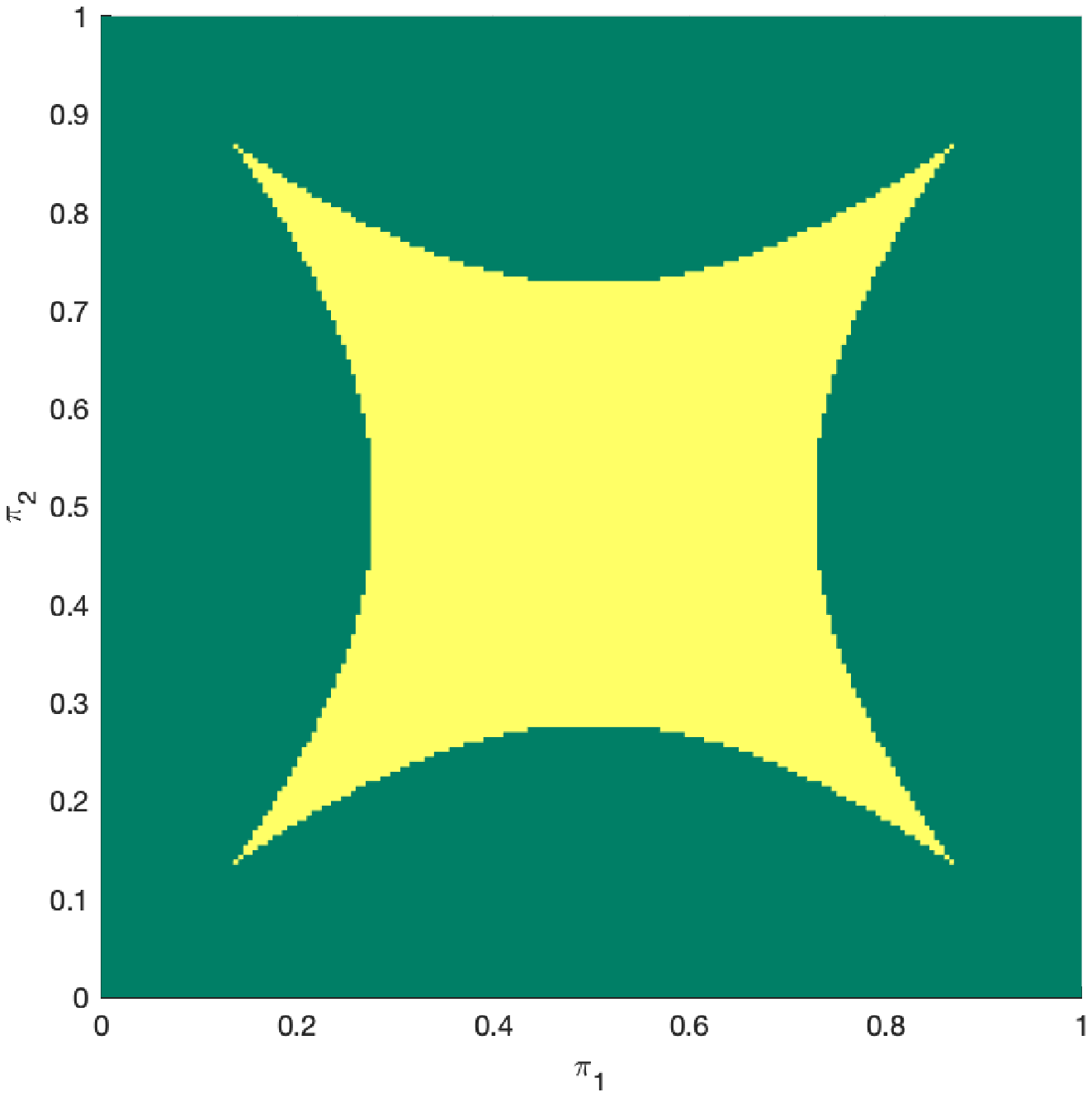}\label{figureST3-3}}
\caption{Continuation regions (yellow) in (ST3) for different values of the cost reduction parameter $\gamma$.
The parameters are $\mu_1=\mu_2=1$ and $c=0.2$; the cost reduction parameter is $\gamma=0.2$, $\gamma=0.5$ and $\gamma=0.8$, respectively.}
\end{figure}

\begin{remark}
A few estimates on the stopping/continuation regions in (ST3) are readily obtained. Let $(A^{\mu,c},1-A^{\mu,c})$ be the continuation 
region in $ST(\mu,c)$. 

If $\frac{1}{2} \leq \gamma < 1$, then 
\[u^{\mu,(1-\gamma)c}(\pi_1)+u^{\mu,(1-\gamma)c}(\pi_2)\leq V\leq g.\]
Consequently, the continuation region is contained in the square
$(A^{\mu,(1-\gamma)c}, 1-A^{\mu,(1-\gamma)c})^2$.

If $0< \gamma \leq \frac{1}{2}$, then 
\[u^{\mu,c/2}(\pi_1) + u^{\mu,c/2}(\pi_2)\leq V\leq u^{\mu,(1-\gamma)c}(\pi_1)+u^{\mu,(1-\gamma)c}(\pi_2).\]
Thus \[(A^{\mu,(1-\gamma)c},1-A^{\mu,(1-\gamma)c})^2\subseteq \C\subseteq \{\pi:\pi_1\mbox{ or }\pi_2\in (A^{\mu,c/2},1-A^{\mu,c/2})\}.\]
In particular, if $\gamma=1/2$, then $V(\pi)=u^{\mu,c/2}(\pi_1) +u^{\mu,c/2}(\pi_2)$ and $\C=(A^{\mu,c/2},1-A^{\mu,c/2})^2$.

\end{remark}

\section{Quickest detection problems}
\label{5}

In this section we provide structural results for the multi-dimensional quickest detection problems (QD1)-(QD3). For the sake of graphical illustrations, we present the results for $n=2$.

\subsection{(QD1)}
Consider the stopping problem
\[V(\pi) = \inf_{\tau} \E_\pi\left[g(\Pi_{\tau})+\int_0^{\tau}h(\Pi_s)ds  \right]\]
where $g(\pi ) = (1-\pi_1)(1-\pi_2)$ and $h(\pi) = c(1-(1-\pi_1)(1-\pi_2))$.

\begin{proposition}
There exists a non-increasing lower semi-continuous function $b: [0, 1]\to [0, 1]$ such that 
\begin{equation}
\label{C}
\C = \{\pi\in[0,1]^2: \pi_2<b(\pi_1)\}.
\end{equation}
\end{proposition}

\begin{proof}
We first note that boundary points $(\pi_1,1)$ and $(1,\pi_2)$ are stopping points since $V= g=0$ 
at such points. Consequently, by unilateral concavity it follows that if a point $(\pi_1,\pi_2)\in\C$, then also 
$[0,\pi_1]\times[0,\pi_2]\subseteq\C$. The existence of a non-increasing function $b: [0, 1]\to [0, 1]$ such that \eqref{C} holds
thus follows; the lower semi-continuity of $b$ is a direct consequence of the continuity of $V$.
\end{proof}

For a graphical illustration, see Figure~\ref{QD1}.

\subsection{(QD2)}

We now study the stopping problem \eqref{V} with $g(\pi)=1-\pi_1\pi_2$ and $h(\pi)=c\pi_1\pi_2$.
Let $[0,B^*_i)$ be the continuation region for the one dimensional problem $QD(\mu_i,\lambda_i,c)$, $i=1,2$.

\begin{proposition}
There exists a non-increasing lower semi-continuous function $b:[B^*_1,1]\to[B^*_2,1]$ such that 
\[\C=\{\pi\in[0,1]^2:\pi_2<b(\pi_1)\}\cup [0,B^*_1)\times[0,1].\]
\end{proposition}

\begin{figure}[h]
  \subfigure[]{\includegraphics[width=0.3\textwidth]{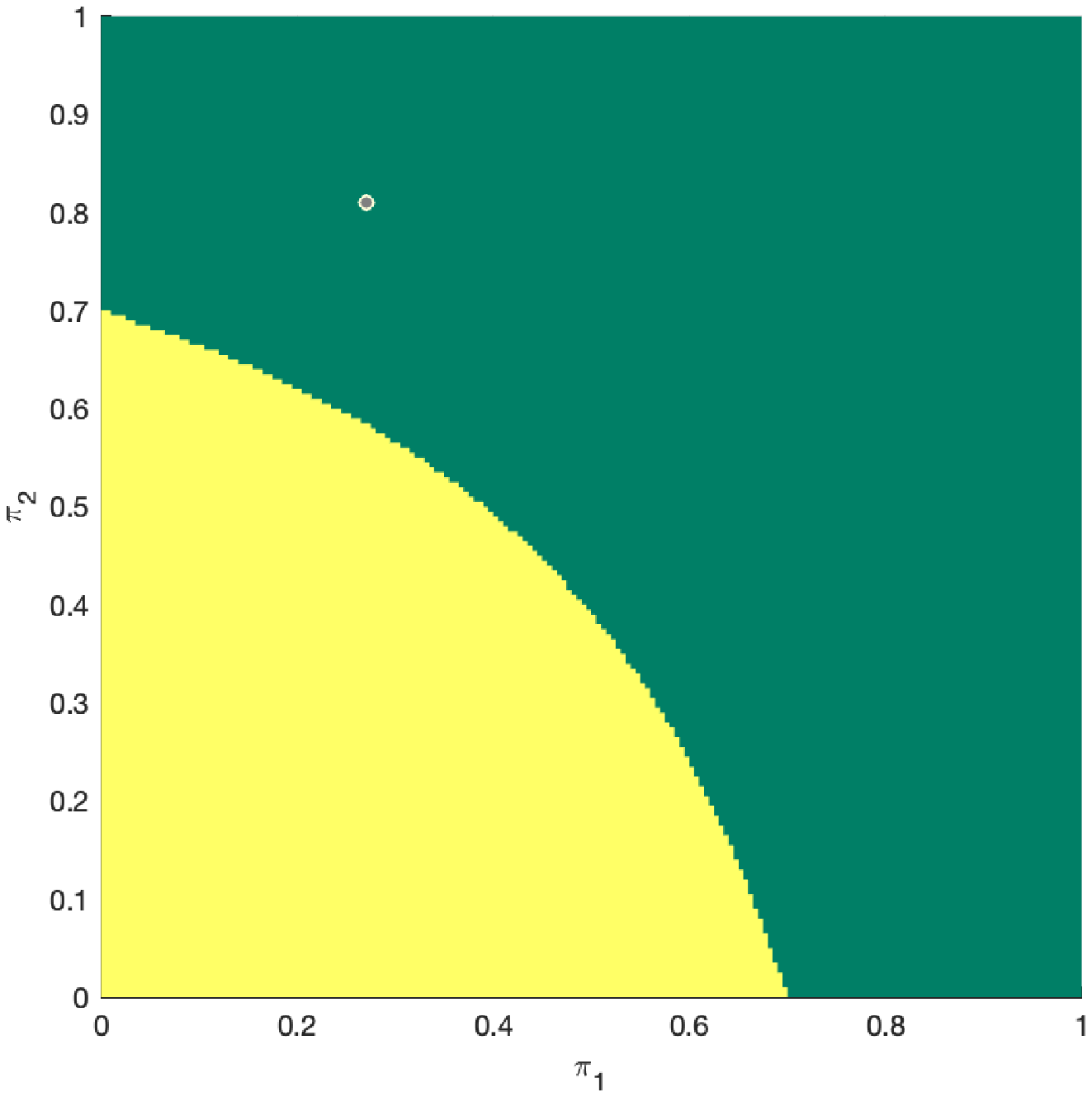}\label{QD1}}
  \subfigure[]{\includegraphics[width=0.3\textwidth]{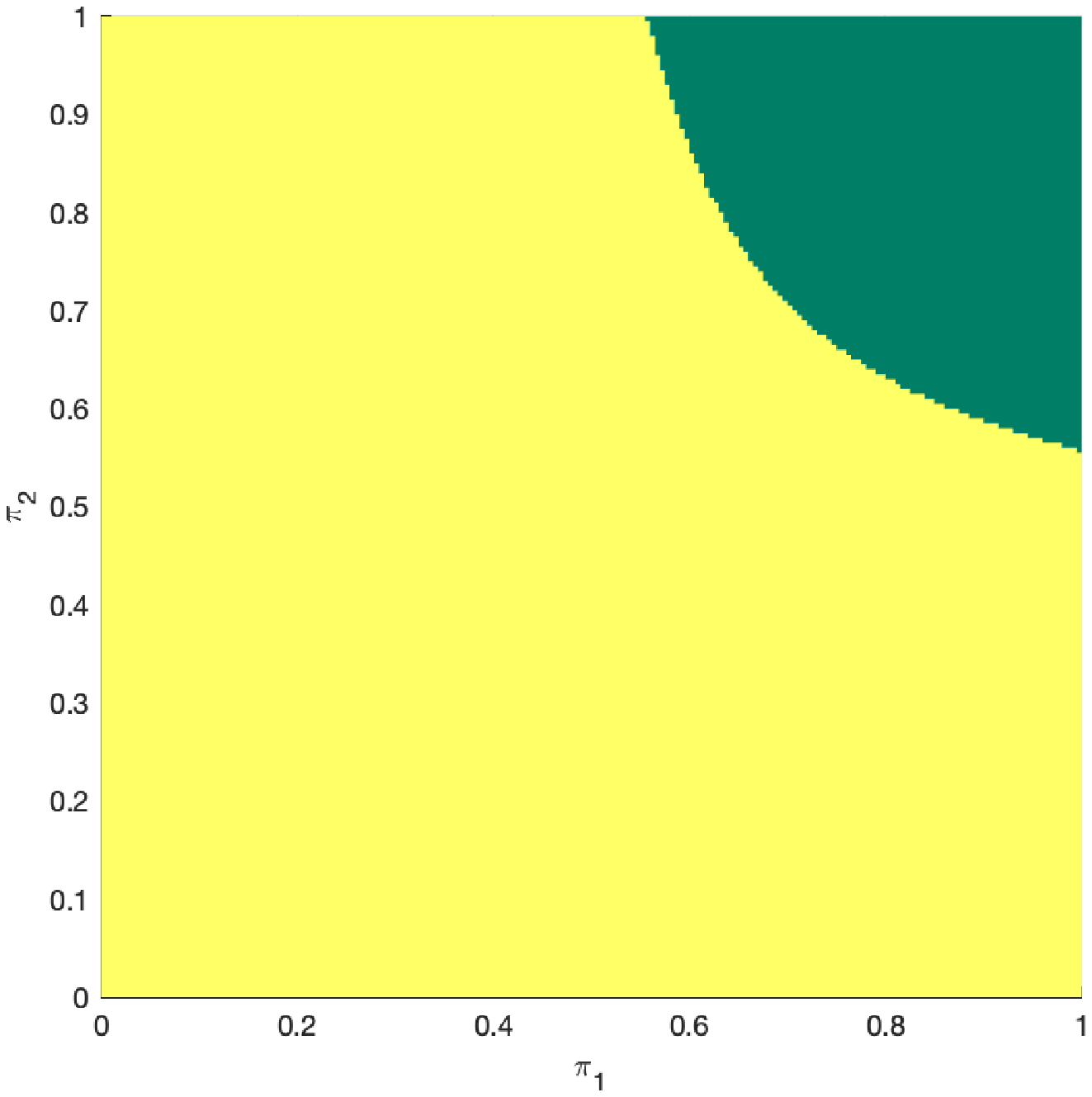}\label{QD2}}
 \subfigure[]{\includegraphics[width=0.3\textwidth]{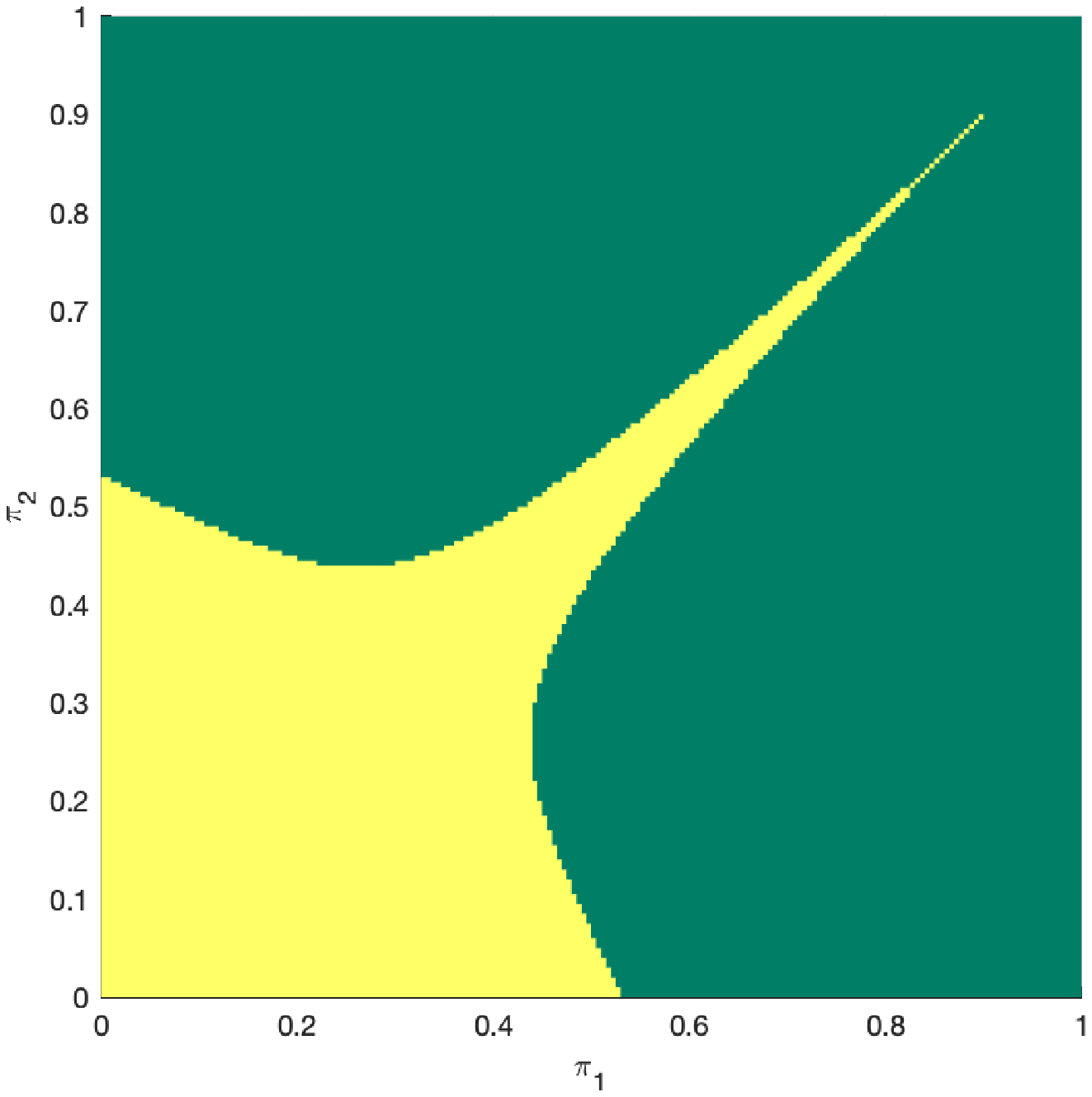}\label{QD3}}
\caption{Continuation regions (yellow) in (QD1) on the left, (QD2) in the middle and in (QD3) on the right.
The parameters are $\mu_1=\mu_2=1$ and $c=1$.}
\end{figure}

\begin{proof}
We first note that the line segment $[B^*_1,1]\times\{1\}$ belongs to the stopping region and that $g$ is affine in the $\pi_2$-direction. Consequently, concavity yields the existence of a function $b:[B^*_1,1]\to[0,1]$ so that
$\C\cap\{\pi_1\geq B^*_1\} =\{\pi_1\geq B^*_1,\pi_2<b(\pi_1\}$.

We now claim that $\{\pi_2<B^*_2\}\subseteq\C$. To see this, let 
\[H(\pi_1,\pi_2):=\pi_1 u_2(\pi_2) + 1-\pi_1\]
where $u_2:=u^{\mu_2,\lambda_2,c}$ is the value function in $QD(\mu_2,\lambda_2,c)$.
Then $H=g$ if $\pi_2\geq B_2^*$, and 
\[\mathcal LH+ c\pi_1\pi_2 = -(1-u_2(\pi_2))\lambda_1(1-\pi_1)\leq 0\]
on $[0,1]\times [0,B_2^*)$, where the inequality is strict provided $\pi_1<1$.
Now consider the stopping time
$\tau_2:=\inf\{t\geq 0:\Pi^2_t\geq B_2^*\}$ which is optimal in $QD(\mu_2,\lambda_2,c)$.
Then
\begin{eqnarray*}
V(\pi_1,\pi_2) &\leq& \E_{\pi}\left[g(\Pi_{\tau_2}^1,\Pi_{\tau_2}^2) + c\int_0^{\tau_2}\Pi^1_t\Pi^2_t\,dt\right]\\
&=& \E\left[H(\Pi_{\tau_2}^1,\Pi_{\tau_2}^2) + c\int_0^{\tau_2}\Pi^1_t\Pi^2_t\,dt\right] \leq H(\pi_1,\pi_2)
\end{eqnarray*}
by supermartingality. Moreover, if $\pi_2<B_2^*$ and $\pi_1<1$, then the second inequality is strict. 
Thus, if $\pi_2<B_2^*$ then 
\[V(\pi_1,\pi_2)\leq H(\pi_1,\pi_2) \leq 1-\pi_1\pi_2=g(\pi_1,\pi_2),\]
where the first inequality is strict if $\pi_1<1$, and the second inequality is strict if $\pi_1>0$. Consequently,  
$\{\pi_2<B^*_2\}\subseteq\C$.

Using $\{\pi_2<B^*_2\}\subseteq\C$, it follows that
$b\geq B^*_2$; interchanging $\pi_1$ and $\pi_2$ shows that $\{\pi_1<B^*_1\}\subseteq\C$ and that $b$ is non-increasing. Finally, 
the continuity of $V$ implies lower semi-continuity of $b$.
\end{proof}

The continuation region in (QD2) is illustrated in Figure~\ref{QD2}.

\subsection{(QD3)}

Now assume that $g(\pi)=1-\pi_1\vee\pi_2$ and 
$h(\pi)=c(\pi_1+\pi_2)$.  
By symmetry, it suffices to describe the structure of the continuation region in $T:=\{\pi\in[0,1]^2:\pi_1\leq \pi_2\}$.

\begin{proposition}
There exists a function $b:[0,1]\to[0,1]$ with $b(\pi_1)\geq \pi_1$ such that 
\begin{equation}
\label{relat}
\C\cap T=\{\pi\in T:\pi_2<b(\pi_1)\}.
\end{equation}
Moreover, $b$ is lower semi-continuous and first non-increasing and then non-decreasing.
\end{proposition}

\begin{proof}
We first note that $g(\pi_1,1)=0$, so $[0,1]\times\{1\}\subseteq \D$, and that $g$ is affine in $\pi_2$ on $T$; concavity thus implies the existence of $b$ such that 
\eqref{relat} holds. Moreover, $g$ is affine also in $\pi_1$ on $T$, so horizontal sections of the stopping region inside $T$ are intervals.
Consequently, the function $b$ is first non-increasing 
and then non-decreasing. Lower semi-continuity follows from the continuity of $V$.
\end{proof}



\begin{thebibliography}{}

\bibitem{BDK}
Bayraktar, E., Dayanik, S. and Karatzas, I. Adaptive Poisson disorder problem. 
{\em Ann. Appl. Probab}. 16 (2006), no. 3, 1190-1261.

\bibitem{BP}
Bayraktar, E. an Poor, V.
Quickest detection of a minimum of two Poisson disorder times. 
{\em SIAM J. Control Optim.} 46 (2007), no. 1, 308-331.

\bibitem{DPS}
Dayanik, S., Poor, V., and Sezer, S. Multisource Bayesian sequential change detection.
{\em Stochastics} 80 (2008), no. 1, 19-50.

\bibitem{DA}
De Angelis, T., Federico, S. and Ferrari, G.
Optimal boundary surface for irreversible investment with stochastic costs. 
{\em Math. Oper. Res.} 42 (2017), no. 4, 1135-1161.

\bibitem{EJT}
Ekstr\"om, E., Janson, S. and Tysk, J. 
Superreplication of options on several underlying assets. 
{\em J. Appl. Probab}. 42 (2005), no. 1, 27-38.

\bibitem{EV}
Ekstr\"om, E. and Vaicenavicius, J. Bayesian sequential testing of the drift of a Brownian motion. 
{\em ESAIM Probab. Stat.} 19 (2015), 626-648.

\bibitem{EV2}
Ekstr\"om, E. and Vaicenavicius, J. 
Monotonicity and robustness in Wiener disorder detection. {\em Sequential Anal.} 38 (2019), no. 1, 57-68.

\bibitem{EJS}
El Karoui, N. Jeanblanc-Picqu\'{e}, M. and Shreve, S. 
Robustness of the Black and Scholes formula. 
{\em Math. Finance} 8 (1998), no. 2, 93-126.

\bibitem{EPZ}
Ernst, P. A., Peskir, G., Zhou, Q.  Optimal real-time detection of a drifting Brownian motion. 
{\em Ann. Appl. Probab.} 30 (2020), no. 3, 1032-1065.


\bibitem{H}
Hobson, D. Volatility misspecification, option pricing and superreplication via coupling. 
{\em Ann. Appl. Probab.} 8 (1998), no. 1, 193-205.

\bibitem{JT}
Janson, S. and Tysk, J.
Volatility time and properties of option prices. 
{\em Ann. Appl. Probab.} 13 (2003), no. 3, 890-913.

\bibitem{JT2}
Janson, S. and Tysk, J. 
Preservation of convexity of solutions to parabolic equations. 
{\em J. Differential Equations} 206 (2004), no. 1, 182-226.


\bibitem{MS}
Muravlev, A.A. and Shiryaev, A.N. (2014). Two sided disorder problem for a Brownian motion in a Bayesian setting.
{\em Proc. Steklov Inst. Math.} 287 (2014), no. 1, 202-224.

\bibitem{PS}
Peskir, G. and Shiryaev, A. {\em Optimal stopping and free-boundary problems.} Lectures in Mathematics ETH Z\"urich. Birkh\"auser Verlag, Basel, 2006.

\bibitem{PS1}
Peskir, G. and Shiryaev, A.
Sequential testing problems for Poisson processes. 
{\em Ann. Statist.} 28 (2000), no. 3, 837-859.

\bibitem{PS2}
Peskir, G. and Shiryaev, A.
Solving the Poisson disorder problem. {\em Advances in finance and stochastics}, 295-312, Springer, Berlin, 2002. 

\bibitem{S}
Shiryaev, A. Two problems of sequential analysis.
{\em Cybernetics} 3 (1967), no. 2, 63-69 (1969).


\bibitem{ZS}
Zhitlukhin, M., Shiryaev, A.N. (2011). A Bayesian sequential testing problem of three hypotheses for Brownian motion.
{\em Stat. Risk Model.} 28 (2011), no. 3, 227-249.
\end{thebibliography}
\end{document}